\documentclass[11pt,a4paper]{amsart}%amsart
\usepackage{amssymb, amsthm, enumerate, amsmath, mathrsfs}

\usepackage{tikz}
\usetikzlibrary{calc,patterns,through}

\theoremstyle{plain}
\newtheorem{thm}{Theorem}[section]
\newtheorem{lemma}[thm]{Lemma}

\newtheorem{claim}{}[thm] %Claim
\newenvironment{subproof}{\begin{proof}[Subproof.]}{\end{proof}}
\DeclareMathOperator{\fcl}{fcl}

\newcommand{\del}{\hspace{-0.5pt}\backslash}

\begin{document}
 \title{Displaying tangles and non-sequential separations}

\author{Ben Clark}
\address{Department of Mathematics, Louisiana State University, Baton Rouge, Louisiana, USA}
\email{bclark@lsu.edu}

\date{\today}

\begin{abstract}
We show that, for any graph or matroid, there is a tree that simultaneously distinguishes its maximal tangles, and, for each maximal tangle $\mathcal{T}$ that satisfies an additional robustness condition, displays all of the non-trivial separations relative to $\mathcal{T}$, up to a natural equivalence. %in a tree-like way.
\end{abstract}

\maketitle

\section{Introduction}

Tangles provide a means of locating the highly-connected pieces of a graph, matroid or, more generally, a connectivity system. Roughly speaking, a tangle of order $k$ does this by choosing a small side of each separation of order strictly less than $k$ in a compatible way. We say that $(X,Y)$ is a \textit{distinguishing separation} for a pair of tangles if they choose different small sides of $(X,Y)$. Tangles were introduced by Robertson and Seymour in \cite{robertson1991graph}, where they showed each graph $G$ has a tree decomposition that displays a minimum-order distinguishing separation for each pair of maximal tangles of $G$. This result was a key ingredient in their graph minors project, and the analogous result for connectivity systems, and hence matroids, was proved by Geelen, Gerards, and Whittle in \cite{geelen2009tangles}.

% Geelen, Gerards, and Whittle \cite{geelen2009tangles} showed that each connectivity system $K$ has a tree that displays a minimum-order distinguishing separation for each pair of maximal tangles in $K$. This extended a key result of Robertson and Seymour's graph minors project \cite{robertson1991graph}. 

%Since $\mathcal{T}$ chooses a small side of each separation of order strictly less than $k$, a 

Consider a tangle $\mathcal{T}$ of order $k$ in a graph or matroid, and suppose that $(X,Y)$ is a $k$-separation such that neither $X$ nor $Y$ is contained in a small side of $\mathcal{T}$. We can think of $(X,Y)$ as a minimum-order separation that gives information about the structure of the highly-connected piece located by $\mathcal{T}$. Given such a separation $(X,Y)$, we say $X$ is \textit{sequential} with respect to $\mathcal{T}$ if there is an ordered partition $(Z_1,\ldots, Z_n)$ of $X$ such that, for each $i\in \{1,2,\ldots, n\}$, the set $Z_i$ is contained in a small side of $\mathcal{T}$ and $Y\cup Z_1\cup\cdots \cup Z_i$ is $k$-separating. We say that $(X,Y)$ is \textit{non-sequential} with respect to $\mathcal{T}$ if neither $X$ nor $Y$ is sequential with respect to $\mathcal{T}$. Clark and Whittle \cite{clark2011tangles} showed that, if $\mathcal{T}$ satisfies an additional necessary robustness condition, then, up to a natural equivalence, the $k$-separations that are non-sequential relative to $\mathcal{T}$ can be displayed in a tree-like way. This extended earlier results on the structure of $3$-separations in $3$-connected matroids \cite{oxley2004structure} and the structure of $4$-separations in $4$-connected matroids \cite{aikin2012structure}.  %and $4$ and $4$

In \cite{clark2011tangles}, the authors asked if it were possible to associate a tree with a graph or matroid that simultaneously displays a distinguishing separation for each pair of maximal tangles, and, for each maximal tangle, displays a representative of each non-sequential separation relative to the tangle in a tree-like way. The main result of this paper, Theorem \ref{main}, shows that this is indeed possible.  

In Section \ref{s2}, we provide the necessary preliminaries and give a precise statement of the main theorem. In Section \ref{s3}, we show that we can work with a connectivity system that allows us to handle the problem of crossing separations. Finally, in Section \ref{s4}, we prove the main theorem.

\section{Preliminaries and the main theorem}\label{s2}

Let $\lambda$ be an integer-valued function on the subsets of a finite set $E$. We call $\lambda$ \textit{symmetric} if $\lambda(X)=\lambda(E-X)$ for all $X\subseteq E$. We call $\lambda$ \textit{submodular} if $\lambda(X)+\lambda(Y)\geq \lambda(X\cup Y)+\lambda(X\cap Y)$ for all $X,Y\subseteq E$. If $\lambda$ is integer-valued, symmetric, and submodular, then $\lambda$ is called a \textit{connectivity function} on $E$. If $E$ is a finite set and $\lambda$ is a connectivity function on $E$, then the pair $(E,\lambda)$ is a \textit{connectivity system}. %Throughout this paper we let $K=(E,\lambda)$ be a connectivity function.

The results in this paper are stated for connectivity systems. Of course, we are primarily interested in connectivity systems because they arise naturally from matroids and graphs. Let $M$ be a matroid on ground set $E$ with rank function $r$. For $X\subseteq E$, we let $\lambda_M(X)=r(X)+r(E-X)-r(M)+1$. It is straightforward to prove that $(E,\lambda_M)$ is a connectivity system. Let $G$ be a graph with edge set $E$. For $X\subseteq E$, we let $\lambda_G(X)$ denote the number of vertices of $G$ that are incident with both an edge of $X$ and an edge of $E-X$. It is also straightforward to prove that $(E,\lambda_G)$ is a connectivity system. Thus for the results in this paper, we immediately obtain corollaries for matroids and graphs.%Moreover, if $G$ is a connected graph, then $\lambda_{M(G)}(X)\leq \lambda_G(X)$.

Let $K=(E,\lambda)$ be a connectivity system, and let $k$ be a positive integer. A partition $(X,E-X)$ of $E$ is called a \textit{$k$-separation} of $K$ if $\lambda(X)\leq k$. A subset $X$ of $E$ is said to be \textit{$k$-separating} in $\lambda$ if $\lambda(X)\leq k$. Two separations $(A,B)$ and $(C,D)$ \textit{cross} if all the intersections $A\cap C$, $A\cap D$, $B\cap C$, and $B\cap D$ are non-empty. The fact that separations can cross is the main obstacle to overcome to obtain this main result.

We now build towards a precise statement of the main result of \cite{geelen2009tangles}, which is essential  in the proof of the main theorem of this paper.

A \textit{tangle of order $k$} in $K=(E,\lambda)$ is a collection $\mathcal{T}$ of subsets of $E$ such that the following properties hold:
\begin{enumerate}
 \item[(T1)] $\lambda(A)<k$ for all $A\in \mathcal{T}$.
 \item[(T2)] If $(A,E-A)$ is a $(k-1)$-separation, then $\mathcal{T}$ contains $A$ or $E-A$.
 \item[(T3)] If $A,B,C\in \mathcal{T}$, then $A\cup B\cup C\not= E$.
 \item[(T4)] $E-\{e\}\notin \mathcal{T}$ for each $e\in E$.
\end{enumerate}

Let $k\geq 1$ be an integer, and let $\mathcal{T}$ be a tangle of order $k$ in a connectivity system $K=(E,\lambda)$. Now let $j\leq k$, and let $\mathcal{T}|_j\subseteq \mathcal{T}$ be the set of $A\in \mathcal{T}$ such that $\lambda(A)\leq j$.

\begin{lemma}
\cite[Lemma 4.1]{geelen2009tangles}
\label{truncate}
$\mathcal{T}|_j$ is a tangle of order $j$ in $(E,\lambda)$.
\end{lemma}

We say that $\mathcal{T}|_j$ is the \textit{truncation} of $\mathcal{T}$ to order $j$. If $\mathcal{T}_1$ and $\mathcal{T}_2$ are tangles, neither of which is a truncation of the other, then there is some \textit{distinguishing separation} $(X_1,X_2)$ with $X_1\in \mathcal{T}_1$ and $X_2\in \mathcal{T}_2$. 

A \textit{tree-decomposition} of $E$ is a tree $T$ with $V(T)=\{1,2,\ldots,n\}$ and a partition $(P_1,P_2,\ldots, P_n)$ of $E$ (where $P_1,P_2,\ldots, P_n$ are called \textit{bags}, and may be empty). Let $e$ be an edge of $T$, and let $T_1$ and $T_2$ be the components of $T\del e$. The separation $(\bigcup_{i\in V(T_1)} P_i, \bigcup_{i\in V(T_2)} P_i)$ of $K$ is said to be \textit{displayed by e}. A separation of $K$ is \textit{displayed by $T$} if it is displayed by some edge of $T$. We can now state the main result of \cite{geelen2009tangles}.

\begin{thm}
 \label{treeoftangles}\cite[Theorem 9.1]{geelen2009tangles}
Let $K=(E,\lambda)$ be a connectivity system, and let $\mathcal{T}_1, \ldots \mathcal{T}_n$ be tangles in $K$, none of which is a truncation of another. Then there exists a tree-decomposition $T$ of $\lambda$ such that $V(T)=\{1,2,\ldots,n\}$ and such that the following hold:
\begin{enumerate}
 \item[(a)] For each $i\in V(T)$ and $e\in E(T)$ if $T'$ is the component of $T\del e$ containing $i$, then the union of those bags that label vertices of $T'$ is not a member of $\mathcal{T}_i$.
 \item[(b)] For each pair of distinct vertices $i$ and $j$ of $T$, there exists a minimum-order distinguishing separation for $\mathcal{T}_i$ and $\mathcal{T}_j$ that is displayed by $T$. 
\end{enumerate}
\end{thm}

We now introduce the terminology necessary to state the main result of \cite{clark2011tangles}. In what follows, unless explicitly stated otherwise, we let $K=(E,\lambda)$ be a connectivity system, and let $\mathcal{T}$ be a tangle of order $k$ in $K$.

A subset $X$ of $E$ is \textit{$\mathcal{T}$-strong} if it is not contained in a member of $\mathcal{T}$; otherwise $X$ is \textit{$\mathcal{T}$-weak}. It is easy to see that supersets of $\mathcal{T}$-strong sets are $\mathcal{T}$-strong, and that subsets of $\mathcal{T}$-weak sets are $\mathcal{T}$-weak. A partition $(X_1,\ldots, X_n)$ of $E$ is \textit{$\mathcal{T}$-strong} if $X_i$ is a $\mathcal{T}$-strong set for all $i\in \{1,2,\ldots,n\}$. In particular, a $k$-separation $(X,E-X)$ of $\lambda$ is \textit{$\mathcal{T}$-strong} if both $X$ and $E-X$ are $\mathcal{T}$-strong sets. We note that if a partition $\{X,E-X\}$ of $E$ is $\mathcal{T}$-strong, then neither $X$ nor $E-X$ is a member of $\mathcal{T}$, so $\lambda(X)\geq k$ by (T2). Thus a $\mathcal{T}$-strong $k$-separation $(X,E-X)$ is exact.

A $\mathcal{T}$-strong $k$-separating set $X$ is \textit{fully closed with respect to $\mathcal{T}$} if $X\cup Y$ is not $k$-separating for every non-empty $\mathcal{T}$-weak set $Y\subseteq E-X$. Let $X$ be a $\mathcal{T}$-strong $k$-separating set. Then the intersection of all fully-closed $k$-separating sets that contain $X$, which we denote by $\fcl_{\mathcal{T}}(X)$, is called the \textit{full closure of $X$ with respect to $\mathcal{T}$}. We use the full closure with respect to $\mathcal{T}$ to define equivalence of $k$-separations. Let $(X,Y)$ and $(X',Y')$ be $\mathcal{T}$-strong $k$-separations of $\lambda$. Then $(X,Y)$ is $\mathcal{T}$-\textit{equivalent} to $(X',Y')$ if $\{\fcl_{\mathcal{T}}(X),\fcl_{\mathcal{T}}(Y)\}= \{\fcl_{\mathcal{T}}(X'),\fcl_{\mathcal{T}}(Y')\}$. We say that $E-X$ is $\mathcal{T}$-\textit{sequential} if $\fcl_{\mathcal{T}}(X)=E$. A $k$-separation $(X,Y)$ is $\mathcal{T}$-\textit{sequential} if $X$ or $Y$ is a $\mathcal{T}$-sequential $k$-separating set; otherwise $(X,Y)$ is \textit{non-sequential with respect to $\mathcal{T}$}. A \textit{partial $k$-sequence for $X$} is a sequence $(X_i)_{i=1}^m$ of pairwise disjoint, non-empty $\mathcal{T}$-weak subsets of $E-X$ such that $X\cup (\bigcup_{i=1}^j X_i)$ is $k$-separating for all $j\in \{1,2,\ldots,m\}$. A partial $k$-sequence $(X_i)_{i=1}^m$ for $X$ is said to be \textit{maximal} if $X\cup (\bigcup_{i=1}^m X_i)$ is inclusion-wise maximal. A maximal partial $k$-sequence is simply called a \textit{$k$-sequence}. We use the following characterisation full closure with respect to $\mathcal{T}$.   

\begin{lemma}
\cite[Lemma 3.6]{clark2011tangles} 
\label{maxsequence}
Let $\mathcal{T}$ be a tangle of order $k$ in a connectivity system $(E,\lambda)$. Let $X$ be a $\mathcal{T}$-strong $k$-separating set, and let $(X_i)_{i=1}^m$ be a partial $k$-sequence for $X$. Then $\fcl_{\mathcal{T}}(X)=X\cup (\bigcup_{i=1}^m X_i)$ if and only if $(X_i)_{i=1}^m$ is a $k$-sequence.
\end{lemma}

We use the following lemmas on $\mathcal{T}$-equivalence.

\begin{lemma}
\cite[Lemma 3.7]{clark2011tangles} 
 \label{3.3}
Let $\mathcal{T}$ be a tangle of order $k$ in a connectivity system $K$, and let $(A, B)$ and $(C, D)$ be two $k$-separations of $K$ that are non-sequential with respect to $\mathcal{T}$. Then $(A, B)$ is $\mathcal{T}$-equivalent to $(C, D)$ if and only if either $\fcl_{\mathcal{T}}(A)= \fcl_{\mathcal{T}}(C)$ or $\fcl_{\mathcal{T}}(A)= \fcl_{\mathcal{T}}(D)$.
\end{lemma}

%equivalence and sequential lemmas for partial $k$-sequences.

\begin{lemma}
\cite[Lemma 3.8]{clark2011tangles} 
 \label{elemseq}
Let $\mathcal{T}$ be a tangle of order $k$ in $K$. Let $(R,G)$ be a $k$-separation of $K$ that is non-sequential with respect to $\mathcal{T}$, and let $A\subseteq G$ be a non-empty $\mathcal{T}$-weak set. If $R\cup A$ is a $k$-separating set, then $(R,G)$ is $\mathcal{T}$-equivalent to $(R\cup A, G-A)$.
\end{lemma}

\begin{lemma}
\cite[Corollary 3.5]{clark2011tangles}
 \label{fclcontaincor}
Let $\mathcal{T}$ be a tangle of order $k$ in a connectivity system $(E,\lambda)$, and let $X$ be a $\mathcal{T}$-strong $k$-separating set. If $(X_i)_{i=1}^m$ is a partial $k$-sequence for $X$, then $\fcl_{\mathcal{T}}(X\cup (\bigcup_{i=1}^m X_i))=\fcl_{\mathcal{T}}(X)$. 
\end{lemma}

In order to handle separations that cross, we need the following notion. A $\mathcal{T}$-strong partition $(P_1,\ldots, P_n)$ of $E$ is a \textit{$k$-flower} in $\mathcal{T}$ with \textit{petals} $P_1,\ldots, P_n$ if, for all $i$, both $P_i$ and $P_i\cup P_{i+1}$ are $k$-separating sets, where all subscripts are interpreted modulo $n$. A $k$-separating set $X$ or $k$-separation $(X,E-X)$ is said to be \textit{displayed} by $\Phi$ if $X$ is a union of petals of $\Phi$.

Let $\pi$ be a partition of $E$ (note that we allow members of $\pi$ to be empty.) Let $T$ be a tree such that every member of $\pi$ labels a vertex of $T$ (some vertices may be unlabelled and no vertex is multiply labelled.) We say that $T$ is a \textit{$\pi$-labelled tree}. The vertices of $T$ labelled by the members of $\pi$ are called \textit{bag vertices}, and the members of $\pi$ are called \textit{bags}. Let $T'$ be a subtree of $T$. The union of those bags that label vertices of $T'$ is the subset of $E$ \textit{displayed} by $T'$. Let $e$ be an edge of $T$. The \textit{partition of $E$ displayed by $e$} is the partition displayed by the connected components of $T\del e$. Let $v$ be a vertex of $T$ that is not a bag vertex. Then the \textit{partition of $E$ displayed by $v$} is the partition displayed by the connected components of $T-v$. The edges incident with $v$ are in natural one-to-one correspondence with the connected components of $T-v$, and hence with the members of the partition of $E$ displayed by $v$. If a cyclic ordering is imposed on the edges incident with $v$, then we cyclically order the members of the partition of $E$ displayed by $v$ in the corresponding order. We say that $v$ is a \textit{$k$-flower vertex for $\mathcal{T}$} if the partition $(P_1,\ldots,P_n)$ of $E$ displayed by $v$, in the cyclic order corresponding to the cyclic order on the edges incident with $v$, is a $k$-flower in $\mathcal{T}$. The $k$-separations displayed by the $k$-flower corresponding to a $k$-flower vertex are called the $k$-separations \textit{displayed by $v$}. A $k$-separation is \textit{displayed} by $T$ if it is displayed by an edge or a $k$-flower vertex of $T$.

We say that $\mathcal{T}$ is a \textit{robust tangle of order $k$} in $K=(E,\lambda)$ if $\mathcal{T}$ is a tangle of order $k$ that satisfies:
\begin{enumerate}
 \item[(RT3)] If $A_1,A_2,\ldots,A_8\in \mathcal{T}$, then $A_1\cup A_2\cup \cdots \cup A_8\neq E$.
\end{enumerate}

The following result is the main theorem of \cite{clark2011tangles}.

\begin{thm}
\cite[Theorem 7.1]{clark2011tangles}
\label{treeofseparations}
Let $\mathcal{T}$ be a robust tangle of order $k$ in a connectivity system $K=(E,\lambda)$. There is a $\pi$-labelled tree such that every $k$-separation of $K$ that is non-sequential with respect to $\mathcal{T}$ is $\mathcal{T}$-equivalent to some $k$-separation displayed by $T$.
\end{thm}

We can now state the main theorem of this paper.

\begin{thm}
\label{main}
Let $K=(E,\lambda)$ be a connectivity system, and let $\mathcal{T}_1, \ldots \mathcal{T}_n$ be tangles in $K$, none of which is a truncation of another. Then there is a $\pi$-labelled tree $T$ such that the following hold.
 \begin{enumerate}
  \item[(i)] For all $i,j\in \{1,\ldots, n\}$ with $i\neq j$, there is a minimum-order distinguishing separation for $\mathcal{T}_i$ and $\mathcal{T}_j$ that is displayed by some edge of $T$; and
    \item[(ii)] For all $i\in \{1,\ldots, n\}$, if $\mathcal{T}_i$ is a robust tangle of order $k$ and $(X,Y)$ is a $k$-separation of $K$ that is non-sequential with respect to $\mathcal{T}_i$, then there is some $k$-separation $(X',Y')$ of $K$ displayed by $T$ that is $\mathcal{T}_i$-equivalent to $(X,Y)$.
  \end{enumerate}
\end{thm}

\section{New tangles from old}\label{s3}

For tangles $\mathcal{T}_1, \ldots \mathcal{T}_n$ in a connectivity system $K$, none of which is a truncation of another, crossing separations are the main obstacle to combining the separations displayed by the tree of Theorem \ref{treeoftangles} and the tree of Theorem \ref{treeofseparations} into a single tree. To overcome this, we will show that for each robust tangle $\mathcal{T}_i$ we can construct the tree of Theorem \ref{treeofseparations} for $\mathcal{T}_i$ in such a way that the distinguishing separations displayed by the tree of Theorem \ref{treeoftangles} conform. We do this by moving to a new connectivity system and tangle, which we describe in this section.

The following construction is found in \cite{geelen2009tangles}. Let $K=(E,\lambda)$ be a connectivity system and let $X\subseteq E$. Let $K\circ X=((E-X)\cup \{x\}, \lambda')$, where for each $A\subseteq E-X$ we let $\lambda'(A)=\lambda(A)$ and $\lambda'(A\cup \{x\})=\lambda(A\cup X)$.

\begin{lemma}
\label{circ}\cite[Lemma 4.2.]{geelen2009tangles}
If $K$ is a connectivity system and $X\subseteq E$, then $K\circ X$ is a connectivity system.
\end{lemma}

Let $X$ be a subset of $E$, and let $\pi=(X_1, \ldots, X_n)$ be a partition of $X$. By repeated application of Lemma \ref{circ}, we deduce that $K\circ \pi:=((E-X)\cup \{x_1,\ldots,x_n\}, \lambda_{\pi})$ is a connectivity system, where for each $A\subseteq E-X$ and $I\subseteq \{1,2,\ldots, n\}$ we let  $\lambda_{\pi}(A\cup (\bigcup_{i\in I} x_i))=\lambda(A\cup (\bigcup_{i\in I} X_i))$.%$\lambda_{\pi}(A)=\lambda(A)$ and

We can also obtain a robust tangle in $K\circ X$ from a robust tangle in $K$ as follows. We omit the routine verification of the robust tangle axioms.

\begin{lemma}
\label{circtangle}
Let $\mathcal{T}$ be a robust tangle of order $k$ in a connectivity system $K$ and let $X$ be in $\mathcal{T}$. Let $\mathcal{T}'$ be the set of subsets of $(E-X)\cup \{x\}$ such that, for all $A\subseteq (E-X)$, we have $A\in \mathcal{T}'$ if and only if $A\in \mathcal{T}$; and $A\cup \{x\}\in \mathcal{T}'$ if and only if $A\cup X\in \mathcal{T}$. Then $\mathcal{T}'$ is a robust tangle of order $k$ in $K\circ X$. 
\end{lemma}

In particular, if $\mathcal{T}$ is a robust tangle of order $k$ in a connectivity system $K=(E,\lambda)$, and $\pi=(X_1,\ldots, X_n)$ is a partition of a subset $X$ of $E$ such that $X_i\in \mathcal{T}$ for all $i\in \{1,2,\ldots,n\}$, then, by repeated application of Lemma \ref{circtangle}, we obtain a robust tangle $\mathcal{T}_{\pi}$ of order $k$ in $K\circ \pi$. For any $A\subseteq E-X$ and $I\subseteq \{1,2,\ldots, n\}$, we have $A\cup (\bigcup_{i\in I} x_i)\in \mathcal{T}_{\pi}$ if and only if $A\cup (\bigcup_{i\in I} X_i)\in \mathcal{T}$.

Let $K=(E,\lambda)$ be a connectivity system, let $\mathcal{T}$ be a robust tangle of order $k$ in $K$, and let $B$ be a subset of $E$. A partition $\pi=(B_1,\ldots, B_n)$ of $B$ is called a \textit{bag partition of $B$ with respect to $\mathcal{T}$} if 
\begin{enumerate}
 \item[(i)] $B_1,\ldots, B_n\in \mathcal{T}$; and
 \item[(ii)] For each $i\in \{1,2,\ldots,n\}$, if $(X,Y)$ is a partition of $B_i$, then $\lambda(X)\geq \lambda(B_i)$ or $\lambda(Y)\geq \lambda(B_i)$.
\end{enumerate}
The sets $B_1,\ldots, B_n$ are called \textit{bags} of $\pi$. By Lemma \ref{circtangle} there is a robust tangle $\mathcal{T}_{\pi}$ of order $k$ in $K\circ \pi$ obtained from $\mathcal{T}$; we call $\mathcal{T}_{\pi}$ the tangle in $K\circ \pi$ \textit{induced} by $\mathcal{T}$. For the remainder of this section we look at $k$-separations in $K$ and $K\circ \pi$, and we show that $\mathcal{T}$- and $\mathcal{T}_{\pi}$-equivalence are compatible. 

Let $B\subseteq E$, and let $(B_1,\ldots, B_n)$ be a bag partition for $B$ with respect to $\mathcal{T}$, and suppose that $(R,G)$ is a $k$-separation that does not cross $B_i$ for all $i\in \{1,2,\ldots,n\}$. Then $(R,G)=(R'\cup (\bigcup_{i\in I} B_i),G'\cup (\bigcup_{i\in (\{1,2,\ldots,n\}-I)} B_i))$ for some partition $(R',G')$ of $E-B$ and some $I\subseteq \{1,2,\ldots,n\}$. It follows immediately from the definition of $K\circ \pi$ that the corresponding partition $(R'\cup (\bigcup_{i\in I} b_i),G'\cup (\bigcup_{i\in (\{1,2,\ldots,n\}-I)} b_i))$ of $(E-B)\cup \{b_1,\ldots, b_n\}$ is a $k$-separation of $K\circ \pi$. We say that the $k$-separation $(R'\cup (\bigcup_{i\in I} b_i),G'\cup (\bigcup_{i\in (\{1,2,\ldots,n\}-I)} b_i))$ of $K\circ \pi$ is \textit{induced} by the $k$-separation $(R,G)$ of $K$. Moreover, it follows immediately from the definition of the tangle $\mathcal{T}_{\pi}$ induced by $\mathcal{T}$ that $(R,G)$ is $\mathcal{T}$-strong if and only if the induced $k$-separation $(R'\cup (\bigcup_{i\in I} b_i),G'\cup (\bigcup_{i\in (\{1,2,\ldots,n\}-I)} b_i))$ is $\mathcal{T}_{\pi}$-strong. We next show that every $\mathcal{T}$-equivalence class of non-sequential $k$-separations contains a member that does not cross any bags of $\pi$.

\begin{lemma}
\label{sepeq}
Let $\mathcal{T}$ be a tangle of order $k$ in a connectivity system $K=(E,\lambda)$, and let $\pi=(B_1,\ldots, B_n)$ be a bag partition for $B$ with respect to $\mathcal{T}$. If $(R,G)$ is a $k$-separation of $K$ that is non-sequential with respect to $\mathcal{T}$, then $(R,G)$ is $\mathcal{T}$-equivalent to a $k$-separation $(R',G')$ that does not cross $B_i$ for any $i\in \{1,2,\ldots,n\}$.
\end{lemma}

\begin{proof}
Assume that $(R,G)$ crosses the minimum number of bags amongst all $\mathcal{T}$-equivalent $k$-separations. We may assume, up to relabelling the bags, that $(R,G)$ crosses $B_1$. By definition of the bag partition $\pi$ we may assume that $\lambda(R\cap B_1)\geq \lambda(B_1)$. Then $R\cup B_1$ is $k$-separating by submodularity, so $(R\cup B_1,G-B_1)$ is a $k$-separation that is $\mathcal{T}$-equivalent to $(R,G)$ by Lemma \ref{elemseq}, and $(R\cup B_1,G-B_1)$ crosses fewer bags than $(R,G)$, a contradiction.
\end{proof}

%By Lemma \ref{sepeq}, if $(R,G)$ is a $k$-separation of $K$ that is non-sequential with respect to $\mathcal{T}$, then there a $k$-separation of $K$ that is $\mathcal{T}$-equivalent to $(R,G)$ with an induced $k$-separation of the connectivity system $K\circ \pi$. 
For $k$-separations that are non-sequential with respect to $\mathcal{T}$, the induced $k$-separations are also non-sequential with respect to the induced tangle $\mathcal{T}_{\pi}$, which we now show.

\begin{lemma}
\label{circtangleeq}
Let $\mathcal{T}$ be a tangle of order $k$ in a connectivity system $K=(E,\lambda)$, and let $\pi=(B_1,\ldots, B_n)$ be a bag partition for $B$ with respect to $\mathcal{T}$. If $(R,G)$ is a $k$-separation of $K$ that is non-sequential with respect to $\mathcal{T}$ and does not cross $B_i$ for any $i\in \{1,2,\ldots,n\}$, then the induced $k$-separation $(R_{\pi}, G_{\pi})$ of $K\circ \pi$ is non-sequential with respect to $\mathcal{T}_{\pi}$. 
\end{lemma}

\begin{proof}
Since $(R,G)$ does not cross any of the $B_i$, we know that $(R,G)=(R'\cup (\bigcup_{i\in I} B_i), G'\cup (\bigcup_{i\in (\{1,2,\ldots,n\}-I)} B_i))$ for some partition $(R', G')$ of $E-B$ and some subset $I$ of $\{1,2,\ldots,n\}$. Hence the induced $k$-separation of $K\circ \pi$ is $(R_{\pi},G_{\pi})=(R'\cup (\bigcup_{i\in I} b_i), G'\cup (\bigcup_{i\in (\{1,2,\ldots,n\}-I)} b_i))$. Assume that $(R_{\pi},G_{\pi})$ is $\mathcal{T}_{\pi}$-sequential. Then by Lemma \ref{maxsequence}, up to switching $R_{\pi}$ and $G_{\pi}$, there is a $k$-sequence $(X_i)_{i=1}^m$ for $R_{\pi}$ such that $R_{\pi}\cup (\bigcup_{i=1}^m X_i)=(E-B)\cup \{b_1,\ldots, b_n\}$. Now, for each $i\in \{1,2,\ldots,m\}$, we have $X_i=X_i'\cup (\bigcup_{i\in J} b_i)$ for some $X_i'\subseteq G'$ and $J\subseteq (\{1,2,\ldots,n\}-I)$. Then, by definition of $K\circ \pi$ and $\mathcal{T}_{\pi}$, it follows that there is a $k$-sequence $(Y_i)_{i=1}^m$ for $R$ where $Y_i=X_i'\cup (\bigcup_{i\in J} B_i)$ for each $i\in \{1,2,\ldots,m\}$, and $R\cup (\bigcup_{i=1}^m Y_i)=E$, so $(R,G)$ is $\mathcal{T}$-sequential by Lemma \ref{maxsequence}, a contradiction.
\end{proof}

%Thus, if $(R,G)$ is a $k$-separation of $K$ that is non-sequential with respect to $\mathcal{T}$, then the corresponding $k$-separation of $K\circ \pi$ is non-sequential with respect to $\mathcal{T}_{\pi}$. 
To conclude this section, we show that $\mathcal{T}$-equivalence can be recovered from $\mathcal{T}_{\pi}$-equivalence.

\begin{lemma}
\label{eqcorrespondence}
Let $\mathcal{T}$ be a tangle of order $k$ in a connectivity system $K=(E,\lambda)$, and let $\pi=(B_1,\ldots, B_n)$ be a bag partition for $B$ with respect to $\mathcal{T}$. Let $(R,G)$ be a $k$-separation of $K$ that is non-sequential with respect to $\mathcal{T}$ and does not cross $B_i$ for all $i\in \{1,2,\ldots,n\}$, and let $(R_{\pi},G_{\pi})$ be the induced $k$-separation of $K\circ \pi$. If $(R_{\pi},G_{\pi})$ is $\mathcal{T}_{\pi}$-equivalent to $(R_{\pi}',G_{\pi}')$, then $(R,G)$ is $\mathcal{T}$-equivalent to the $k$-separation $(R',G')$ that induces $(R_{\pi}',G_{\pi}')$.
\end{lemma}

\begin{proof}
We may assume, up to switching $R_{\pi}'$ and $G_{\pi}'$, that $\fcl_{\mathcal{T}_{\pi}}(R_{\pi})= \fcl_{\mathcal{T}_{\pi}}(R_{\pi}')$. Let $(Y_i)_{i=1}^m$ be a $k$-sequence for $R_{\pi}$ and $(Z_i)_{i=1}^{p}$ be a $k$-sequence for $R_{\pi}'$. Then, for each $i\in \{1,2,\ldots,m\}$ and $j\in \{1,2,\ldots,p\}$, we have $Y_i=Y_i'\cup (\bigcup_{i\in I} b_i)$ and $Z_j=Z_j'\cup (\bigcup_{j\in J} b_j)$ for some subsets $Y_i'$ and $Z_j'$ of $E-B$ and some $I, J\subseteq \{1,2,\ldots,n\}$. For each $i\in \{1,2,\ldots,m\}$ and $j\in \{1,2,\ldots,p\}$, let $Y_i''=Y_i'\cup (\bigcup_{i\in I} B_i)$ and let $Z_j''=Z_j'\cup (\bigcup_{j\in J} B_j)$. Then $(Y_i'')_{i=1}^m$ is a partial $k$-sequence for $R$, $(Z_i'')_{i=1}^{p}$ is a partial $k$-sequence for $R'$, and $R\cup (\bigcup_{i\in \{1,2,\ldots,m\}} Y_i'')$ is equal to $R'\cup (\bigcup_{i\in \{1,2,\ldots,p\}} Z_i'')$ since $\fcl_{\mathcal{T}_{\pi}}(R_{\pi})= \fcl_{\mathcal{T}_{\pi}}(R_{\pi}')$, so it follows from Lemma \ref{fclcontaincor} that $\fcl_{\mathcal{T}}(R)= \fcl_{\mathcal{T}}(R')$. Hence $(R,G)$ is $\mathcal{T}$-equivalent to the $k$-separation $(R',G')$ by Lemma \ref{3.3}.
\end{proof}

\section{The main theorem}\label{s4}

The next lemma is used to show that minimum-order distinguishing separations give rise to bag partitions.

\begin{lemma}
\label{bag2}
Let $\mathcal{T}_1$ and $\mathcal{T}_2$ be tangles in $K$, and let $(X_1,X_2)$ be a minimum-order distinguishing separation for $\mathcal{T}_1$ and $\mathcal{T}_2$ with $X_1\in \mathcal{T}_1$ and $X_2\in \mathcal{T}_2$. If $(R,G)$ is a partition of $X_2$, then $\lambda(R)\geq \lambda(X_2)$ or $\lambda(R)\geq \lambda(X_2)$. 
\end{lemma}

\begin{proof}
Assume that $\lambda(R)<\lambda(X_2)$ and $\lambda(G)<\lambda(X_2)$. Then, since $(X_1, X_2)$ is a minimum-order distinguishing separation for $\mathcal{T}_1$ and $\mathcal{T}_2$, neither $(R, E-R)$ nor $(G, E-G)$ is a distinguishing separation for $\mathcal{T}_1$ and $\mathcal{T}_2$. Now $G$ and $R$ cannot belong to $\mathcal{T}_1$ and $\mathcal{T}_2$ because $R\cup G\cup X_1=E$, contradicting (T3) for $\mathcal{T}_1$. Thus, up to relabelling $R$ and $G$, we may assume that $E-R$ belongs to $\mathcal{T}_1$ and $\mathcal{T}_2$. But $(E-R)\cup X_2=E$, contradicting (T3) for $\mathcal{T}_2$. 
\end{proof}

We now prove the main theorem, which we restate for convenience.

\begin{thm}
\label{main2}
Let $K=(E,\lambda)$ be a connectivity system, and let $\mathcal{T}_1, \ldots \mathcal{T}_n$ be tangles in $K$, none of which is a truncation of another. Then there is a partition $\pi$ of $E$, and a $\pi$-labelled tree $T$ such that the following hold.
 \begin{enumerate}
  \item[(i)] For all $i,j\in \{1,\ldots, n\}$ with $i\neq j$, there is a minimum-order distinguishing separation for $\mathcal{T}_i$ and $\mathcal{T}_j$ that is displayed by some edge of $T$; and
    \item[(ii)] For all $i\in \{1,\ldots, n\}$, if $\mathcal{T}_i$ is a robust tangle of order $k$ and $(X,Y)$ is a $k$-separation of $K$ that is non-sequential with respect to $\mathcal{T}_i$, then there is some $k$-separation $(X',Y')$ of $K$ displayed by $T$ that is $\mathcal{T}_i$-equivalent to $(X,Y)$.
  \end{enumerate}
\end{thm}

\begin{proof}
By Theorem \ref{treeoftangles} there is a tree $T$ on the vertex set $\{1,2,\ldots, n\}$ with the following properties.

\begin{enumerate}
 \item[(a)] For each $i\in V(T)$ and $e\in E(T)$ if $T'$ is the component of $T\del e$ containing $i$, then the union of those bags that label vertices of $T'$ is not a member of $\mathcal{T}_i$.
 \item[(b)] For each pair of distinct vertices $i$ and $j$ of $T$, there exists a minimum-order distinguishing separation for $\mathcal{T}_i$ and $\mathcal{T}_j$ that is displayed by $T$. 
\end{enumerate}

Let $i\in \{1,2,\ldots,n\}$, and let $N(i)$ be the set of vertices of $T$ that are adjacent to $i$.

\begin{claim}
\label{c1}
If $j\in N(i)$, then the separation displayed by the edge $ij$ of $T$ is a minimum-order distinguishing separation for $\mathcal{T}_i$ and $\mathcal{T}_j$.  
\end{claim}

\begin{subproof}
For each edge $e\in E(T)-\{ij\}$ the vertices $i$ and $j$ are in the same component of $T\del e$, so, by Theorem \ref{treeoftangles} (a), the separation $(R,G)$ displayed by $ij$ is the only separation displayed by $T$ that distinguishes $\mathcal{T}_i$ and $\mathcal{T}_j$. Hence, by Theorem \ref{treeoftangles} (b), the separation $(R,G)$ is of minimum order. 
\end{subproof}

Let $\mathcal{T}_i\in \{\mathcal{T}_1, \ldots \mathcal{T}_n\}$. We associate a $\pi_i$-labelled tree with $\mathcal{T}_i$, where $\pi_i$ is a partition of $E$. If $\mathcal{T}_i$ is not a robust tangle, then we let $T_i$ be the tree with a single vertex labelled by $E$. Suppose that $\mathcal{T}_i$ is a robust tangle of order $k$. Let $B\subseteq E$ be the bag labelling $i$ in the tree $T$ and, for each $j\in N(i)$, let $B_j$ be the set displayed by the component of $T\del ij$ containing $j$. 

\begin{claim}
$\pi_i'=\{B_j\ |\ j\in N(i)\}$ is a bag partition for $E-B$ with respect to $\mathcal{T}_i$. 
\end{claim}

\begin{subproof}
 It follows from Theorem \ref{treeoftangles} (a) and \ref{c1} that the members of $\pi_i'$ belong to $\mathcal{T}_i$. By Lemma \ref{bag2} and \ref{c1}, for each $j\in N(i)$, if $(X,Y)$ is a partition of $B_j$, then $\lambda(X)\geq \lambda(B_j)$ or $\lambda(Y)\geq \lambda(B_j)$.
\end{subproof}

Let $\mathcal{T}_{\pi_i'}$ be the bag tangle of order $k$ in $K\circ \pi_i'$ induced by $\mathcal{T}_i$. Then $\mathcal{T}_{\pi_i'}$ is a robust tangle of order $k$ by Lemma \ref{circtangle}. By Theorem \ref{treeofseparations}, there is a maximal partial $k$-tree $T_{\pi_i'}$ for $\mathcal{T}_{\pi_i'}$. Let $T_i$ be the $\pi_i$-labelled tree obtained from $T_{\pi_i'}$ by replacing $b_j$ by $B_j$ for all $j\in N(i)$.

\begin{claim}
\label{maxktreei}
Every $k$-separation of $K$ that is non-sequential with respect to $\mathcal{T}_i$ is $\mathcal{T}_i$-equivalent to some $k$-separation displayed by $T_i$. 
\end{claim}

\begin{subproof}
 Let $(R,G)$ be a $k$-separation of $K$ that is non-sequential with respect to $\mathcal{T}_i$. We may assume, by Lemma \ref{sepeq}, that $(R,G)$ does not cross any bag of $\pi_i'$. Then the $k$-separation $(R_{\pi_i'}, G_{\pi_i'})$ of $K\circ \pi_i'$ induced by $(R,G)$ is non-sequential with respect to $\mathcal{T}_{\pi_i'}$ by Lemma \ref{circtangleeq}. By Theorem \ref{treeofseparations}, the $k$-separation $(R_{\pi_i'}, G_{\pi_i'})$ is $\mathcal{T}_{\pi_i'}$-equivalent to a $k$-separation $(R_{\pi_i'}', G_{\pi_i'}')$ of $K\circ \pi_i'$ that is displayed by the tree $T_{\pi_i'}$. Then the $k$-separation $(R',G')$ of $K$ that induces $(R_{\pi_i'}', G_{\pi_i'}')$ is displayed by the tree $T_i$, and $(R',G')$ is $\mathcal{T}_i$-equivalent to $(R,G)$ by Lemma \ref{eqcorrespondence}.
\end{subproof}

Now, for each tangle $\mathcal{T}_i$, there is a $\pi_i$-labelled tree $T_i$ such that, if $\mathcal{T}_i$ is robust, then $T_i$ displays, up to $\mathcal{T}_i$-equivalence, all of the separations that are non-sequential with respect to $\mathcal{T}_i$. Moreover, for each $j\in N(i)$, the set $B_j$ displayed by the component of $T\del ij$ not containing $i$ is contained some bag of $\pi_i$. 

Consider the forest $F$ with components $T_1,T_2,\ldots, T_n$. We construct a $\pi$-labelled tree $\tau$ from $F$ by performing the following procedure for each edge of $T$. For each edge $ij$ of $T$ and $k\in\{i,j\}$, let $X_k$ be the set displayed by the component of $T\del ij$ that contains the vertex $k$. By the construction of $T_i$ and $T_j$, there is some vertex $v$ of $T_i$ labelled by a bag $B_v$ that contains $X_j$, and some vertex $w$ of $T_j$ labelled by a bag $B_w$ that contains $X_i$. We join $v$ and $w$ by an edge, then relabel $v$ by the bag $B_v-X_j$ and relabel $w$ by the bag $B_w-X_i$. Let $\tau$ denote the resulting tree. 

It follows from the construction that $\tau$ is a $\pi$-labelled tree. For each distinct $i,j\in \{1,2,\ldots,n\}$, since there is some minimum-order distinguishing separation for $\mathcal{T}_i$ and $\mathcal{T}_j$ displayed by an edge of $T$, it follows from the construction that this separation is also displayed by an edge of $\tau$, so (i) holds. Moreover, if $\mathcal{T}_i$ is a robust tangle of order $k$, then every $k$-separation of $K$ that is non-sequential with respect to $\mathcal{T}_i$ is $\mathcal{T}_i$-equivalent to some $k$-separation displayed by $T_i$ by \ref{maxktreei}, and by the construction this separation is displayed by $\tau$.
\end{proof}

\bibliographystyle{acm}
\bibliography{ta2}

\begin{thebibliography}{1}

\bibitem{aikin2012structure}
{\sc Aikin, J., and Oxley, J.}
\newblock The structure of the 4-separations in 4-connected matroids.
\newblock {\em Advances in Applied Mathematics 48}, 1 (2012), 1--24.

\bibitem{clark2011tangles}
{\sc Clark, B., and Whittle, G.}
\newblock {Tangles, trees, and flowers}.
\newblock {\em Journal of Combinatorial Theory, Series B 103\/} (2013),
  385--407.

\bibitem{geelen2009tangles}
{\sc Geelen, J., Gerards, B., and Whittle, G.}
\newblock {Tangles, tree-decompositions and grids in matroids}.
\newblock {\em Journal of Combinatorial Theory, Series B 99}, 4 (2009),
  657--667.

\bibitem{oxley2004structure}
{\sc Oxley, J., Semple, C., and Whittle, G.}
\newblock {The structure of the 3-separations of 3-connected matroids}.
\newblock {\em Journal of Combinatorial Theory, Series B 92}, 2 (2004),
  257--293.

\bibitem{robertson1991graph}
{\sc Robertson, N., and Seymour, P.}
\newblock {Graph minors. X. Obstructions to tree-decomposition}.
\newblock {\em Journal of Combinatorial Theory, Series B 52}, 2 (1991),
  153--190.

\end{thebibliography}
\end{document}